\documentclass[11pt, oneside]{amsart}
\usepackage[margin = 1.25in]{geometry}
\usepackage{tabularx, hyperref}
\usepackage{amssymb} \usepackage{amsfonts} \usepackage{amsmath}
\usepackage{amsthm} \usepackage{epsfig, subfig}
\usepackage{ amscd, amsxtra, latexsym}
\usepackage[all]{xy}
\usepackage{caption}
\usepackage{enumerate}
\usepackage{color}

\newtheorem{lemma}{Lemma}[section]
\newtheorem{thm}[lemma]{Theorem}
\newtheorem{prop}[lemma]{Proposition}
\newtheorem{cor}[lemma]{Corollary}
\newtheorem{conj}[lemma]{Conjecture}

\theoremstyle{definition}
\newtheorem{defn}[lemma]{Definition}

\theoremstyle{definition}
\newtheorem*{claim}{Claim}

\definecolor{darkgreen}{cmyk}{1,0,1,.2}

\newcommand{\R} {\ensuremath {\mathbb{R}}}

\newcommand{\Z} {\ensuremath {\mathbb{Z}}}

\newcommand{\Mod}{\mathrm{Mod}}
\newcommand{\Out}{\mathrm{Out}}
\newcommand{\C}{\mathcal{C}}

\newcommand{\matE} {\ensuremath {\mathbb{E}}}
\newcommand{\matP} {\ensuremath {\mathbb{P}}}

\address{Department of Mathematics, ETH Zurich, 8092 Zurich, Switzerland}
\email{sisto@math.ethz.ch}

\address{Department of Mathematics, Temple University, Philadelphia, PA, USA}
\email{samuel.taylor@temple.edu}

\begin{document}

\title[Largest projections and shortest curves]{Largest projections for random walks and \linebreak shortest curves in random mapping tori}
\author{Alessandro Sisto and Samuel J. Taylor}
\maketitle

\begin{abstract}
We show that the largest subsurface projection distance between a marking and its image under the $n$th step of a random walk grows logarithmically in $n$, with probability approaching $1$ as $n \to \infty$. Our setup is general and also applies to (relatively) hyperbolic groups and to $\Out(F_n)$.

We then use this result to prove Rivin's conjecture that for a random walk $(w_n)$ on the mapping class group, the shortest geodesic in the hyperbolic mapping torus $M_{w_n}$ has length on the order of $1/ \log^2(n)$.
\end{abstract}

\section{Introduction}
In \cite{Maher}, Maher proved that a random walk generated by a nonelementary measure on the mapping class group $\Mod(S)$ has \emph{positive drift} with respect to the action $\Mod(S) \curvearrowright \C(S)$ on the curve complex $\C(S)$ of an orientable surface $S$. Informally, this means that for a fixed curve $\alpha \in \C(S)$, the displacement $d_{\mathcal C(S)}(\alpha,w_n\alpha)$ between $\alpha$ and its image under the $n$th step of the random walk $(w_n)$ grows linearly in $n$ with probability approaching $1$ as $n \to \infty$. Generalizations were made to actions of countable groups on hyperbolic spaces by Calegari--Maher \cite{calegari2015statistics} and Maher--Tiozzo \cite{MaherTiozzo}, see also \cite{MathieuSisto}.

In this note, we give finer information about the behavior of random walks on mapping class groups, by studying largest subsurface projections. Subsurface projections are very useful for example in view of the distance formula \cite{MM2} for mapping class groups, as well as their connections with the geometry of hyperbolic $3$--manifolds, see e.g. \cite{Minsky-bounded_geometry} and Section \ref{sec:shortest} below. We show that the largest distance between subsurface projection of $1$ and $w_n$ grows logarithmically in $n$ with probability approaching $1$ as $n \to \infty$. Since curve complex distance is ``stalled'' while progress is being made in a subsurface, this complements Maher's result. 

In fact, our methods apply to several other examples, which we list in the theorem below.

\begin{thm}\label{thm:examples}
 Consider one of the following setups:
 \begin{itemize}
 \item Let $G$ be hyperbolic group, $Q$ an infinite index quasiconvex subgroup, and fix a finite generating set $S$ for $G$ which contains a generating set for $Q$. Let $\{\mathcal C(Y)\}_{Y\in\mathcal S}$ be the set of induced subgraphs of left cosets of $Q$ in the Cayley graph $\mathrm{Cay}_S(G)$. For $Y\in\mathcal S$, let $\pi_Y:G\to\mathcal C(Y)$ denote a closest point projection in $\mathrm{Cay}_S(G)$.
  \item Let $G$ be the mapping class group of a closed connected oriented surface of genus at least $2$, and let $\{\mathcal C(Y)\}_{Y\in\mathcal S}$ be the set of curve complexes of either all proper subsurfaces or all proper subsurfaces of a given topological type. For $Y\in\mathcal S$, let $\pi_Y:G\to \mathcal C(Y)$ denote the subsurface projection, i.e. we let $\pi_Y(g)$ be the subsurface projection to $Y$ of $g\lambda$ for a fixed marking $\lambda$ \cite{MM2}.
  \item Let $G$ denote $\Out(F_n)$, $n\geq 3$, and let $\{\mathcal C(Y)\}_{Y\in\mathcal S}$ be the set of free factor complexes of either all proper free factors of $F_n$ or all proper free factors of a given rank $\ge 2$. For $Y\in\mathcal S$, let $\pi_Y:G\to\mathcal C(Y)$ denote the subfactor projection, i.e. $\pi_Y(g)$ is the subfactor projection to the factor complex of $Y$ of $g\lambda$ for a fixed marked graph $\lambda$ \cite{BFproj, Taylproj}.
  
  \item Let $G$ be hyperbolic relative to its proper subgroups $H_1,\dots,H_n$, each containing an undistorted element\footnote{We call an $g$ undistorted if $n\mapsto g^n$ is a quasi-isometric embedding. One can weaken the assumption to $H_i$ being infinite, but that requires adjustments to the proof, and we opted to keep the proof as simple and as uniform as possible.}. Fix a finite generating set $S$ for $G$ which contains generating sets for each $H_i$,  and let $\{\mathcal C(Y)\}_{Y\in\mathcal S}$ be the set of induced subgraphs of
left cosets of $H_1$ in $\mathrm{Cay}_S(G)$.
 For $Y\in\mathcal S$, let $\pi_Y:G\to\mathcal C(Y)$ denote a closest point projection in $\mathrm{Cay}_S(G)$.
 \end{itemize}

 Let $\mu$ be a symmetric measure whose support is finite and generates $G$, and let $(w_n)$ denote the random walk driven by $\mu$. Then there exists $C$ so that
 $$\matP\left(\sup_{Z\in\mathcal S} d_{\C(Z)}(\pi_Z(1),\pi_Z(w_n))\in [C^{-1} \log n, C\log n]\right)\to 1,$$
 as $n$ tends to infinity. 
\end{thm}

We will treat all cases simultaneously by using common features of the mentioned setups. We explain those conditions in Section \ref{sec:statement}, where we also state the theorem that covers all cases (Theorem \ref{thm:main}).

In fact, we remark that Theorem \ref{thm:main} covers even more cases. For example, it can be applied in the contexts of hyperplanes in cube complexes, of hierarchically hyperbolic spaces \cite{behrstock2014hierarchically, behrstock2015hierarchically}, and of hyperbolically embedded subgroups \cite{DGO}. Also, there is no need for the support of the measure to generate $G$, as long as it generates a ``large enough'' semigroup of $G$ (in the mapping class group case, it suffices that the semigroup contains a pseudo-Anosov and an infinite order reducible element). In the interest of brevity, we have chosen  Theorem \ref{thm:examples} as stated.

\subsection*{Application to random fibered $3$--manifolds}
Let $S$ be a closed connected orientable surface of genus at least $2$ and $\Mod(S)$ its mapping class group. For each $f \in \Mod(S)$, we denote the corresponding mapping torus $M(f)$. This is the $3$--manifold constructed by starting with $S \times [0,1]$ and glueing $S \times \{1\}$ to $S \times \{0\}$ via a homeomorphism in the class of $f$. Thurston's celebrated hyperbolization theorem for fibered $3$--manifolds states that $M(f)$ has a (unique by Mostow Rigidity) hyperbolic structure if and only if $f$ is pseudo-Anosov \cite{thurston1998hyperbolic, otal2001hyperbolization}.

The property of being pseudo-Anosov, and hence of determining a hyperbolic mapping torus, is typical. In fact, Rivin \cite{rivin2008walks} and Maher \cite{Maher} proved that for an appropriate random walk $(w_n)$ on $\Mod(S)$ the probability that $w_n$ is pseudo-Anosov goes to $1$ as $n \to \infty$. Hence, it make sense to address questions about the typical geometry of these random hyperbolic $3$-manifolds. For example, in \cite[Conjecture 5.10]{rivin2014statistics}, Rivin conjectures:

\begin{conj}[Rivin] \label{conj:rivin}
The injectivity radius of a random mapping torus of a surface $S$ decays as $1/\log^2 (n)$.
\end{conj}

Rivin also proves his conjecture in the case where $S$ is a punctured torus. In the case of punctured surfaces, injectivity radius should be interpreted as $1/2\cdot \mathrm{sys}(M)$, where $\mathrm{sys}(M)$ is the \emph{systole} of the hyperbolic manifold $M$, i.e. the length of the shortest geodesic in $M$.

In Section \ref{sec:shortest} we prove Rivin's conjecture by establishing:

\begin{thm} \label{thm:shortest_geo}
Let $\mu$ be a symmetric measure whose support is finite and generates $\Mod(S)$, and let $(w_n)$ denote the random walk driven by $\mu$. Then there exists $C>0$ so that 
\[
\matP\left(\frac{C^{-1}}{\log^2 (n)} \le \mathrm{sys}(M(w_n)) \le  \frac{C}{\log^2 (n)}\right)\to 1,
\] 
as $n \to \infty$.
\end{thm}

\subsection*{Outline} In Section \ref{sec:statement} we introduce the general setup that covers all the cases of Theorem \ref{thm:examples}. In Section \ref{sec:expdecay} we prove a result useful to establish both the lower and the upper bound on projections, namely that it is exponentially unlikely that the projection of the random walk to $Z$ moves a large distance.
 The proof is based on ideas from \cite{Si-tracking,Si-contr}, but we simplify the arguments from those papers. In Section \ref{sec:lowerbound} we prove the lower bound on projections.
Besides our application to short curves on random mapping tori, this is the main and most original contribution of this paper. 
The proof is based on the second moment method, and to the best of our knowledge this is the first time that the method is applied in a similar context. In Section \ref{sec:upperbound} we prove the upper bound on projections, using ideas from \cite{ATW} and \cite{Si-tracking} (in which the upper bound is proved in the case of mapping class groups). In Section \ref{sec:examples} we verify that the examples in Theorem \ref{thm:examples} fit into the general framework.

Finally, in Section \ref{sec:shortest} we review the connection between subsurface projections and lengths of curves in hyperbolic mapping tori. When then use this to
prove Theorem \ref{thm:shortest_geo} verifying Rivin's conjecture.

\subsection*{Acknowledgements}
We are grateful to Igor Rivin, as well as one of the referees, for pointing out the connection between our work and Conjecture \ref{conj:rivin}. We also thank Yair Minsky and David Futer for helpful conversations about hyperbolic geometry. Finally, we thank the anonymous referees for valuable suggestions that improved our exposition.

This material is based upon work supported by the National Science Foundation under grant No. DMS-1440140 while the authors were in residence at the Mathematical Sciences Research Institute in Berkeley, California, during the Fall 2016 semester. The second author is also partially supported by NSF DMS-1400498.

\section{Statement of the main theorem}\label{sec:statement}

We now describe the common features of the projections in the various setups of Theorem \ref{thm:examples}.

\begin{defn}\label{defn:proj_system}
 A \emph{projection system}  $(\mathcal S,Y_0,\{\pi_{Z}\}_{Z\in\mathcal S},\pitchfork)$ on a finitely generated group $G$ endowed with the word metric $d_G$ is a quadruple consisting of the following objects.
 \begin{enumerate}
  \item A set $\mathcal S$ together with an action of $G$ of $\mathcal S$ with one orbit, and a specified element $Y_0\in\mathcal S$.
  \item For each $Z\in\mathcal S$, a metric space $\mathcal C(Z)$ and an $L$--Lipschitz map $\pi_Z:(G,d_G)\to \mathcal C(Z)$. We set $d_Z(x,y)=d_{\mathcal C(Z)}(\pi_Z(x),\pi_Z(y))$.
  \item We have $d_{gZ}(gx,gy)=d_Z(x,y)$ for each $g,x,y\in G$ and $Z\in\mathcal S$.
  \item\label{item:Behrstock} An equivariant symmetric relation $\pitchfork$ on $\mathcal S$. We require that there exists $B>0$ so that whenever $Y_0\pitchfork gY_0$, we have $\min\{d_{Y_0}(g,h),d_{gY_0}(1,h)\}< B$ for every $h\in G$.
  \item There exists an integer $s$ so that for each pairwise distinct $Y_1,\dots,Y_s\in\mathcal S$ there exist $i,j$ so that $Y_i\pitchfork Y_j$.
 \end{enumerate}
\end{defn}

Item \ref{item:Behrstock} above is probably most important, and in the context of mapping class groups, where $\C(Z)$ is the curve complex of a subsurface, it is called the Behrstock inequality \cite{Be}. See Figure \ref{Fig1}. The reader may consult Section \ref{sec:mcg} to see how this notion of a projection system fits into the well-known context of subsurface projections for mapping class groups.

It might be worth noting that there is no (Gromov-)hyperbolic space involved, even though we will make use of hyperbolic spaces to check a condition from Definition \ref{defn:admissible} below.

\begin{figure}[htbp]
\begin{center}
\includegraphics[width = .5 \textwidth]{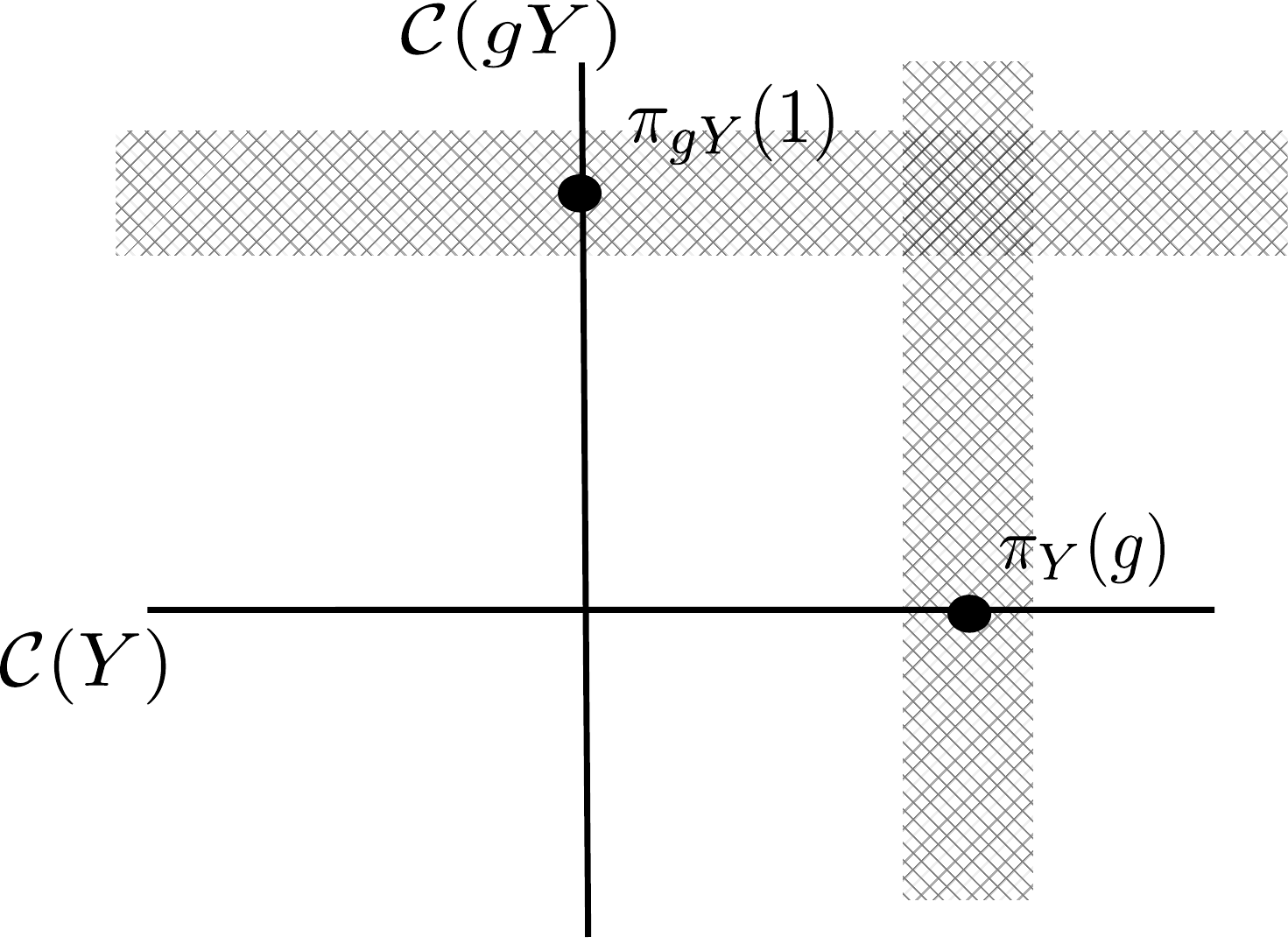}
\caption{An illustration of the Behrstock inequality (similar to Figure $6$ of \cite{Be}). Here, $Y_0 \pitchfork gY_0$ and the shaded region indicates the allowable projection of $h \in G$ to $\C(Y_0) \times \C(gY_0)$.}
\label{Fig1}
\end{center}
\end{figure}

We now give the hypotheses that we will need on the measure. Notice that we did not require the stabilizer of $Y_0$ to act on $\mathcal C(Y_0)$, which is the reason why item \ref{item:far_proj_same_k} below might look strange at first.

\begin{defn}\label{defn:admissible}
 Fix a projection system  $(\mathcal S,Y_0,\{\pi_{Z}\}_{Z\in\mathcal S},\pitchfork)$ on the group $G$. We call a finitely supported measure $\mu$ \emph{semi-admissible} if the random walk $(w_n)$ driven by $\mu$ satisfies the following.
 \begin{enumerate}
\item\label{item:super_transverse} There are $h_1,h_2\in G$ in the support of $\mu$ such that for every $Z\in\mathcal S$ either $h_1Y_0\pitchfork Z$ or $h_2Y_0\pitchfork Z$. 
\item\label{item:far_proj_same_k}  There are $x_1,x_2 \in G$ in the support of $\mu$ such that $d_{Y_0}(x_1h,x_2h)\geq 2B$ for each $h\in G$ and for $B$ as in Definition \ref{defn:proj_system}.\ref{item:Behrstock}.
 \item\label{item:linearproj} There exists $C_0\geq 1$ so that $\matP \big (w_n\in Stab(Y_0), d_{Y_0}(1,w_n)\geq n/C_0 \big )$ is positive for each $n\geq 1$, where $Stab(Y_0)$ denotes the stabilizer of $Y_0$ in $G$.
  \item\label{item:exp_decay_not_overlap} There exists $C_0\geq 1$ so that $\matP(w_nY_0\not\pitchfork Y_0)\leq C_0e^{-n/C_0}$ for each $n  \geq 1$.
 \end{enumerate}
We call a measure $\mu$ \emph{admissible} if both $\mu$ and the reflected measure $\hat{\mu}(g)=\mu(g^{-1})$ are semi-admissible.
\end{defn}

We remark that in our applications, we show that a convolution power of some initial measure is admissible. (See Section \ref{sec:examples}.) The main theorem of the paper is the following.

\begin{thm}\label{thm:main}
 Let $G$ be a group and let $(\mathcal S,Y_0,\{\pi_{Z}\}_{Z\in\mathcal S},\pitchfork)$ be a projection system on $G$. Let $(w_n)$ be a random walk driven by an admissible measure. Then there exists $C\geq 1$ so that
 $$\matP\left( \sup_{Z\in\mathcal S} d_Z(1,w_n) \in [C^{-1} \log n , C\log n]\right)$$
 goes to $1$ as $n$ goes to $\infty$.
\end{thm}

The upper bound and lower bound on random projections are proved separately in Theorem \ref{thm:lower} and Theorem \ref{thm:upper} respectively, which together give Theorem \ref{thm:main}.

\section{Exponential decay of projection distance}\label{sec:expdecay}

In this section we fix the projection system $(\mathcal S,Y_0,\{\pi_{Z}\}_{Z\in\mathcal S},\pitchfork)$ on the finitely generated group $G$ endowed with the word metric $d_G$. Also, we let $\mu$ be a semi-admissible measure generating the random walk $(w_n)$.

The following lemma tells us, thinking of $g$ as an intermediate step of the random walk, that there is a definite probability that the projection onto $Z$ of the random walk does not change from some point on.

\begin{lemma}\label{lem:proj_stabilizes}
There exist $\epsilon,C>0$ so that for any $n$ the following holds. For any $g\in G$ and $Z \in \mathcal{S}$, we have $\matP(d_Z(g,gw_n)\geq C)\leq 1-\epsilon$.
\end{lemma}

The proof exploits a ``replacement'' trick, where we start with $w_n$ having $d_Z(g,gw_n)\geq C$ and (thinking of $w_n$ as a random path) replace its initial subpath of length $2$ to lower the $d_Z$.

\begin{proof}
We first note that it suffice to prove the lemma for $Z = Y_0$. This is because $Z = zY_0$ for some $z \in G$ and $d_Z(g,gw_n) = d_{zY_0}(g,gw_n) = d_{Y_0}(z^{-1}g,z^{-1}gw_n)$.

Fix $h_1,h_2$ as in Definition \ref{defn:admissible}.\ref{item:super_transverse}, and $x_1,x_2$ as in Definition \ref{defn:admissible}.\ref{item:far_proj_same_k}. If $n =1$, the statement of the lemma holds if we choose $C> L\sup_{h\in supp(\mu)}d_G(1,h)$, where $L$ is the Lipschitz constant of $\pi_{Y_0}$.
 For $n\geq 2$ we have
 $$\matP(d_{Y_0}(g,gw_n)\geq C)=\sum_{h\in G}\matP(d_{Y_0}(g,gw_2h)\geq C)\matP(w_2^{-1}w_n=h).$$
 
Fix any $h\in G$. There exists $i$ so that $g^{-1}Y_0\pitchfork h_iY_0$ and hence $Y_0\pitchfork gh_iY_0$. Also, since $d_{gh_iY_0}(gh_ix_1h,gh_ix_2h)\geq 2B$, there exists $j$ so that $d_{gh_iY_0}(1,gh_ix_jh)\geq B$. Hence $d_{Y_0}(gh_i,gh_ix_jh)< B$ by Definition \ref{defn:proj_system}.\ref{item:Behrstock}, so that $d_{Y_0}(g,gh_ix_jh)< B+Ld_G(1,h_i)$.
This proves $\matP(d_{Y_0}(g,gw_2h)\geq C)\leq 1-\epsilon$ if $C>B+L\max\{d_G(1,h_1),d_G(1,h_2)\}$, where we set $\epsilon=\min_{i,j}\{\mu(h_i)\mu(x_j)\}>0$. To finish up,
 
 $$\matP(d_{Y_0}(g,gw_n)\geq C)\leq (1-\epsilon)\sum_{h\in G}\matP(w_2^{-1}w_n=h)=1-\epsilon. $$ 
\end{proof}

%
%
%

The following key proposition says that it is (exponentially) unlikely that a random walk projects far away on $Z$. To prove the lower bound on the largest random projection we just need the case $R=0$, i.e. we do not need the conditional probability.

\begin{prop}\label{prop:exp_decay_proj}
There exists $M\geq 1$ so that for each $Z\in\mathcal S$, each positive integer $n$ and each $t,R\geq 0$ we have
$$\matP \big(d_Z(1,w_n)\geq t+R \;| \; d_Z(1,w_n)\geq R \big)\leq Me^{-t/M}.$$
\end{prop}

The idea of proof is to show that the probability that the projection to $Z$ is much further away than $s \ge0$ plus some constant is a bounded multiple of the probability that it is about $s$. This is because, in view of Lemma \ref{lem:proj_stabilizes}, once an intermediate step of the random walk projects further than $s$, there is a definite probability that the projection does not change.

\begin{proof}
First of all, let us rephrase the statement. Let $f_{Z,n}(s)=\matP(d_Z(1,w_n)\geq s)$. We have to find $M$, not depending on $Z,n$, so that $f_{Z,n}(t+R)\leq f_{Z,n}(R) Me^{-t/M}$ for each $t,R\geq 0$. For the proof, fix $Z,n$ and set $f(s) = f_{Z,n}(s)$.

We now fix some constants. Let $\epsilon,C>0$ be as in Lemma \ref{lem:proj_stabilizes}; in particular 
$$\matP(d_Z(g,gw_m)\geq C)\leq \frac{1-\epsilon}{\epsilon} \: \matP(d_Z(g,gw_m)< C)$$
for each $g\in G$ and $m\geq 0$, since $\matP(d_Z(g,gw_m)< C)\geq \epsilon$. We increase $C$ to ensure that $Ld_G(1,g)\leq C$ for each $g\in \mathrm{supp}(\mu)$, where $L$ is the Lipschitz constant of $\pi_Z$. In particular, any given step of the random walk moves the projection by at most $C$.

For $g\in G$ and $m$ an integer, denote by $E_{g,m}$ the event where $w_{n-m}=g$ and $n-m=\min\{j:d_Z(1,w_j)\geq s+C\}$. Then
\begin{align*}
 &f(s+3C)  \\
 &\le \sum_{\substack{d_Z(1,g)\in [s+C,s+2C]\\ m\leq n}} \matP(d_Z(1,gw_m)\geq s+3C)\matP(E_{g,m}) \\
 &\leq \sum \matP(d_Z(g,gw_m)\geq C)\matP(E_{g,m}) \\
 &\leq \frac{1-\epsilon}{\epsilon}\sum \matP(d_Z(g,gw_m)< C)\matP(E_{g,m}) \\
 &\le \frac{1-\epsilon}{\epsilon}\left(f(s)-f(s+3C)\right),
\end{align*}
where we used $f(s)-f(s+3C)=\matP \big (d_Z(1,w_n)\in [s,s+3C) \big )$.
 
Hence, $f(s+3C)\leq (1-\epsilon) f(s)$, and in turn we get $f(R+3Ci)\leq (1-\epsilon)^if(R)$ for each integer $i\geq 0$. This implies the required exponential decay of $f$.
\end{proof}

\section{Lower bound on projections via the second moment method}\label{sec:lowerbound}

\subsection{Heuristic discussion}
In this section we prove that there exists a projection of at least logarithmic size with high probability. The reason why one expects a logarithmic size projection is that, roughly speaking, a random word of length $n$ contains all subwords of length $\epsilon \log n$, and in particular it will contain a subword that creates a logarithmic size projection. The remaining parts of the word should not affect this projection too much in view of Proposition \ref{prop:exp_decay_proj}. This heuristic alone, however, only gives that the expected number of logarithmic size projections diverges, but it does not say anything about the probability that one exists.

The actual proof relies on the second moment method, i.e. the estimate that for a random variable $X\ge 0$ with finite variance and $\matE(X)>0$, we have
\[
\matP(X >0) \ge \frac{\matE(X)^2}{\matE(X^2)}.
\] 

The second moment method is especially suited for dealing with random variables $X$ that can be written as $\sum_{i=1}^n Y_i$, where ``most pairs'' $Y_i,Y_j$ are ``mostly uncorrelated'', meaning that $\matE(Y_iY_j)$ is approximately $\matE(Y_i)\matE(Y_j)$. In this case the numerator $\sum_{i,j} \matE(Y_i)\matE(Y_j)$ is approximately equal to the denominator $\sum_{i,j} \matE(Y_iY_j)$, implying that $\matP(X>0)$ is close to $1$.

\subsection{The proof}\label{subsec:proof_lower}

We use the notation of Theorem \ref{thm:main} as well as Definitions \ref{defn:proj_system} and \ref{defn:admissible}. In this section we prove that there exists $\epsilon_0>0$ so that
$$\matP\left(\sup_{Z\in\mathcal S} d_Z(1,w_n)\geq \epsilon_0\log n\right)\to 1,\ \ \ (*)$$
as $n$ tends to $\infty.$

%
%

Let $c$ be the minimal probability of an element in $\mathrm{supp} (\mu)$.
Then for each $k$ and each $g\in G$ we have that $\matP(w_k=g)$ is either $0$ or at least $c^k$. We fix a positive $\epsilon_1< \frac{1}{\log(1/c)}$ from now until the end of the section and set $k(n)=\lfloor\epsilon_1\log n\rfloor$.
To simplify the notation, we fix $n$ and set $k=k(n)$, and stipulate that all constants appearing below do not depend on $n$.

In view of the discussion above and Definition \ref{defn:admissible}.\ref{item:linearproj}, there exists $\epsilon_2\in (0,1)$ so that for every sufficiently large $n$ we can choose $x_n\in G$ with the properties that
\begin{enumerate}
 \item $p_n:=\matP(w_{k}=x_n)\geq n^{\epsilon_1 \log(c)}$,
 \item $x_nY_0=Y_0$, and
 \item $d_Y(1,x_{n})\geq \epsilon_2\log(n)$.
\end{enumerate}


For $i\leq n-k$, let $W_i$ be the indicator function of the event $w_i^{-1}w_{i+k}=x_n$. Also, let $L_i$ be the indicator function for the event that $d_{Y_0}(w_i^{-1}, 1) \le \epsilon_2 \log(n)/3$ and let $R_i$ be the indicator function for the event that 
$d_{Y_0}(1, w_{i+k}^{-1}w_n) \le \epsilon_2 \log(n)/3$.

Set $Y_i = L_iW_iR_i$ and note that if $Y_i = 1$, then
\begin{align*}
d_{w_iY_0}(1, w_n) &\ge d_{w_iY_0}(w_i,w_ix_n) - d_{w_iY_0}(1,w_i) - d_{w_ix_nY_0}(w_ix_n, w_n) \\
&\ge d_{Y_0}(1,x_n) - d_{Y_0}(w_i^{-1}, 1) - d_{Y_0}(1, w_{i+k}^{-1}w_n)\\
&\ge \epsilon_2 \log(n) - 2\frac{\epsilon_2}{3} \log(n)\\
&= \frac{\epsilon_2}{3} \log(n).\ \ \ (**)
\end{align*}

Hence, what we want to show is that with high probability there exists $i$ with $Y_i=1$.

\begin{lemma} \label{lem:exp1}
$\matE(Y_i) =p_n(1 - O(n^{-\epsilon_3}))$ for each $i$, where $\epsilon_3>0$.
\end{lemma}

\begin{proof}
 Since $Y_i = L_iW_iR_i$ and $L_i, W_i, R_i$ are independent, it suffices to show that $\matP(L_i = 0)$ and $\matP(R_i = 0)$ decay polynomially as $n \to \infty$. Indeed, by Proposition \ref{prop:exp_decay_proj},
\begin{align*}
\matP\left(d_{Y_0}(w_i^{-1},1) > \frac{\epsilon_2}{3} \log(n)\right) &\le Me^{-\frac{\epsilon_2}{3M}\log(n)}\\
&\le M n^{-\frac{\epsilon_2}{3M}}.
\end{align*}
as required. The case  for $\matP(R_i = 0)$ is similar.
\end{proof}

The following proposition is the key one to apply the second moment method.

\begin{prop} \label{prop:exp2}
 $\matE(Y_iY_j)=p_n^2(1 - O(n^{-\epsilon_4}))$ whenever $|i-j|\geq \log n$, where $\epsilon_4>0$.
\end{prop}

The idea of proof is the following. We have to prove that if in two specified spots along the random path we see the word $x_n$, then it is very likely that $Y_i=Y_j=1$, i.e. that certain projections are not too big. In order to show this, we consider the 3 remaining subpaths of the random path. See Figure \ref{Fig:subpaths}. Such paths and their inverses give small projection to $Y$ by Proposition \ref{prop:exp_decay_proj}. But then it is easy to control all projections we are interested in using the Behrstock inequality.

\begin{proof}
Fix $i\leq j-\log n$. Let $A_1 = \{W_i = 1\}$, $A_2 = \{W_j= 1 \}$, and $A_3$ be the event that either $Y_0 \not\pitchfork w_{i+k}^{-1}  w_jY_0$ or one of the following distances is larger than $\frac{\epsilon_2}{3} \log(n) - B$:
 \begin{enumerate}
  \item $d_{Y_0}(1,w^{-1}_i)$,
  \item $d_{Y_0}(1,w_{i+k}^{-1}w_j)$,
  \item $d_{Y_0}(1,w_j^{-1}w_{i+k})$,
  \item $d_{Y_0}(1,w_{j+k}^{-1}w_n)$.
 \end{enumerate}
Notice that $A_1,A_2,A_3$ are independent. We claim that $W_iW_j1_{A_3^c} \le Y_iY_j \le W_iW_j$. Once we establish the claim, we have that $p_n^2(1 - \matP(A_3)) \le \matE(Y_iY_j) \le p_n^2$. Since the probability of each of the $5$ events making up $A_3$ decays polynomial in $n$, the first by the admissibility condition Definition \ref{defn:admissible}.\ref{item:exp_decay_not_overlap} and the last four by Proposition \ref{prop:exp_decay_proj}, this will complete the proof of the proposition.

\begin{figure}[htbp]
\begin{center}
\includegraphics[width = .9 \textwidth]{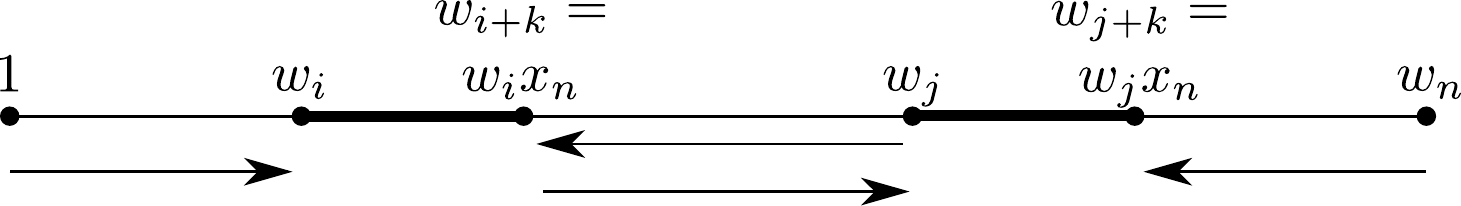}
\caption{The line represents the random path under consideration, with the thicker segments being the occurrences of $x_n$. The arrows suggest the paths that have small enough projection onto $Y_0$ in the event $A_3^c$ (which is often).}
\label{Fig:subpaths}
\end{center}
\end{figure}

To prove the claim, we show that $A_1 \cap A_2 \cap A_3^c \subset \{Y_i =1\} \cap \{Y_j = 1\}$, i.e. that if
 \begin{itemize}
  \item $w_i^{-1}w_{i+k}=x_n$, $w_j^{-1}w_{j+k}=x_n$,
  \item $Y_0 \pitchfork w_{i+k}^{-1}  w_jY_0$, and
  \item all distances in $(1) - (4)$ listed above are at most $\frac{\epsilon_2}{3} \log(n) -B$
 \end{itemize}
then for $l$ equal to $i$ or $j$,  $d_{Y_0}(w_l^{-1}, 1) \le \frac{\epsilon_2}{3} \log(n)$ and
$d_{Y_0}(1, w_{l+k}^{-1}w_n) \le \frac{\epsilon_2}{3} \log(n)$. We show this for $l=i$ since the other case is similar.

First, note that by assumption $w_{i+k} = w_ix_n$ and $w_{j+k} = w_jx_n$. Hence,
\begin{align*}
d_{Y_0}(1, w_{i+k}^{-1}w_n) &= d_{w_iY_0}(w_ix_n, w_n)\\
&\le d_{w_iY_0}(w_ix_n, w_j) + d_{w_iY_0}(w_j, w_n)\\
&\le \frac{\epsilon_2}{3} \log(n) - B  +  d_{w_iY_0}(w_j, w_n),
\end{align*}
and so it suffices to show that $d_{w_iY_0}(w_j, w_n) \le B$. If not, then since $w_ix_nY_0=w_iY_0 \pitchfork w_jY_0$, we must have that $d_{w_jY_0}(w_ix_n,w_n) \le B$. On the other hand, by the triangle inequality,
\begin{align*}
d_{w_jY_0}(w_ix_n,w_n) &\ge d_{w_jY_0}(w_j,w_jx_n) - d_{w_jY_0}(w_ix_n,w_j)-d_{w_jx_nY_0}(w_jx_n,w_n)\\
& =d_{Y_0}(1,x_n)-d_{Y_0}(1,w_{i+k}^{-1}w_j)-d_{Y_0}(1,w_{j+k}^{-1}w_n)\\
&\ge \epsilon_2\log n -2 \frac{\epsilon_2}{3}\log n +2B,
\end{align*}
a contradiction.
\end{proof}

We are ready to prove the lower bound $(*)$.

\begin{thm}\label{thm:lower}
Let $(\mathcal S,Y,\{\pi_{Z}\}_{Z\in\mathcal S},\pitchfork)$ be a projection system on the finitely generated group $G$ and let $\mu$ be a an admissible measure generating the random walk $(w_n)$. Then there exists $\epsilon_0>0$ so that
\[
\matP \left(\sup_{Z\in\mathcal S} d_{Z}(1, w_n) \ge \epsilon_0 \log(n) \right) \to 1,
\]
as $n \to \infty$.
\end{thm}

\begin{proof}
Let $Y_i$ be as above and set $X = X_n = \sum_{i=1}^{n-k} Y_i$. We show that $\matP(X> 0)$ approaches $1$ as $n \to \infty$, which suffices in view of the estimate $(**)$.

Since $\matE(Y_iY_j) \le \matE(Y_i) \le p_n(1 - O(n^{-\epsilon_3}))$ by Lemma \ref{lem:exp1}, Proposition \ref{prop:exp2} implies that

\begin{align*}
\sum_{i,j}\matE(Y_iY_j) &= \sum_{|i-j|< \log n} \matE(Y_iY_j)+\sum_{|i-j|\geq \log n} \matE(Y_iY_j) \\
&\leq 3n\log (n) p_n(1-O(n^{-\epsilon_3}))+ n^2 p_n^2(1-O(n^{-\epsilon_4}))\\
&\leq n^2 p_n^2(1-o(1)),
\end{align*}
where we used that $n\log(n)p_n/ n^2p_n^2 = \log(n)/np_n$ tends to $0$. This holds since $np_n\geq n^{1+\epsilon_1 \log(c)}$, which grows polynomially since $\epsilon_1<\frac{1}{\log(1/c)}$.

By the second moment method (and Lemma \ref{lem:exp1}), we have
\begin{align*}
\matP (X > 0) &\ge \frac{\matE(X)^2}{\matE(X^2)} = \frac{\sum_{i,j}\matE(Y_i)\matE(Y_j)}{\sum_{i,j}\matE(Y_iY_j)} \\
&\geq \frac{(n-k)^2p_n^2(1-o(1))}{n^2 p_n^2(1- o(1))}=1-o(1),
\end{align*}
as required.
\end{proof}

\section{Upper bound on projections via distance formula lower bound}\label{sec:upperbound}

%
%

We start with a proposition that provides a distance-formula-type lower bound on the distance in $G$ between two elements.

This will be useful for us because in order to show that with high probability there is a logarithmic upper bound on $\sup_{Z} d_Z(1,w_n)$ the idea is the following. For each given $Z$ we have the required upper bound in view of Proposition \ref{prop:exp_decay_proj}, and in view of the following proposition we only need to check a controlled number of $Z$.

\begin{prop} \label{lowerbound}
Let $(\mathcal S,Y_0,\{\pi_{Z}\}_{Z\in\mathcal S},\pitchfork)$ be a projection system on the finitely generated group $G$ endowed with the word metric $d_G$.
Then there are $K,C \ge 0$ so that for all $h\in G$ we have
$$\sum_{Z\in\mathcal S} \{\{d_Z(1,h)\}\}_{K} \le  C\cdot  d_G(1,h),$$
where $\{\{x\}\}_K = x$ if $x \ge K$ and $0$ otherwise.
\end{prop}

In the setting of the mapping class groups, the metric spaces $\C(Z)$ are curve complexes of subsurfaces and Proposition \ref{lowerbound} follows from the Masur--Minsky distance formulas \cite{MM2}. However, the proof we give does not rely on the Masur--Minsky hierarchy machinery and applied to our general notion of a projection system. Our proof is a generalization of the argument appearing in \cite{ATW}.

\begin{proof}
We use the notation and constants in Definitions \ref{defn:proj_system}. 

Let $K  = 5B+3L$ and fix a geodesic $1 = h_0 , h_1, \ldots, h_N = h$ in $G$.  Let $\Omega$ be the set of $Z\in\mathcal S$ with $d_Z(1,h)\geq K$. For each $Z \in \Omega$ choose $i_Z , t_Z \in \{0,\ldots,N\}$ as follows: $i_Z$ is the largest index $k$ with $d_Z(h_0,h_k) \le 2B+L$ and $t_Z$ is the smallest index $k$ greater than $i_Z$ with $d_Z(h_k, h_N) \le 2B+L.$  Write $I_Z = [i_Z, t_Z] \subset \{0,1,\ldots, N\}$ and note that this subinterval is well-defined.
Further since the projection of adjacent vertices in the geodesic have $d_Z$ less than or equal to $L$, $d_Z(h_0,h_{k}), d_Z(h_{k},h_N) \ge 2B$ for all $k \in I_Z$ and $d_Z(h_{i_Z},h_{t_Z}) \ge B+L$. 

\begin{claim}
If $aY_0,bY_0 \in \Omega$ and $aY_0 \pitchfork bY_0$ then $I_{aY_0} \cap I_{bY_0}  = \emptyset $.
\end{claim}

\begin{proof}
Toward a contradiction, take $k \in I_{aY_0} \cap I_{bY_0}$. Since $aY_0\pitchfork bY_0$ we have either $d_{aY_0}(b,h_0) < B$ or $d_{bY_0}(a, h_0) < B$. Assume the former; the latter case is proven by exchanging the occurrences of $a$ and $b$ in the proof. By the triangle inequality, 
\begin{eqnarray*}
d_{aY_0}(b, h_k) &\ge& d_{aY_0}(h_0,h_k) - d_{aY_0}(b, h_0) \\
&\ge& 2B - B = B.
\end{eqnarray*}

So, since $aY_0\pitchfork bY_0$, we have $d_{bY_0}(a,h_k) < B$ and 
\begin{eqnarray*}
d_{bY_0}(a, h_N) &\ge& d_{bY_0}(h_k,h_N) - d_{bY_0}(a,h_k) \\
&\ge& 2B -B = B.
\end{eqnarray*}
and we conclude, again since $aY_0 \pitchfork bY_0$, that $d_{aY_0}(b, h_N) \le B$. This, together with our initial assumption, implies
$$d_{aY_0}(h_0,h_N) \le d_{aY_0}(h_0, b) + d_{aY_0}(b, h_N) \le 2B < K$$
contradicting that $aY_0 \in \Omega$.
\end{proof}

Returning to the proof of the proposition, we have the intervals $\{I_Z: Z\in \Omega\}$ of $\{0,1, \ldots, N\}$. By the claim above and our assumption on the number of pair-wise $\pitchfork$-incomparable elements of $\mathcal S$, each $k \in \{0,1, \ldots, N \}$ is contained in at most $s$ intervals. Hence,
$$\sum_{Z \in \Omega} | t_Z - i_Z | \le  s \cdot  d_G(1, h) .$$

Finally, using the Lipschitz condition on the projections,
\begin{eqnarray*}
d_Z(1,h) &\le& d_Z(h_{i_Z}, h_{t_Z}) + 4B +2L \\
&\le& L|t_Z - i_Z| +4B +2L.
\end{eqnarray*}
Since, we have $d_Z(1,h) \ge 5B+3L$ for each $Z \in \Omega$, we get $\frac{1}{5L} \cdot d_Z(1,h)\le |t_Z- i_Z|$ and so putting this with the inequality above
$$\sum_{Z \in \Omega} d_Z(1,h) \le 5sL \cdot d_G(1,h) $$
as required.
\end{proof}

\begin{thm}\label{thm:upper}
Let $(\mathcal S,Y_0,\{\pi_{Z}\}_{Z\in\mathcal S},\pitchfork)$ be a projection system on the finitely generated group $G$ and let $\mu$ be a semi-admissible measure generating the random walk $(w_n)$. Then there exists $A \ge 1$ so that
 $$\matP(\exists Z\in\mathcal S\ :\ d_{Z}(1,w_n)\geq A\log n)$$
 tends to $0$ as $n$ tends to $\infty$.
\end{thm}

\begin{proof}
Let $C,K$ be as in Proposition \ref{lowerbound}, let $M$ be as in Proposition \ref{prop:exp_decay_proj}, and set $l = \sup_{h \in supp (\mu)}d_G(1,h)$. 
Then we notice that $\sum_{Z\in\mathcal S} \matP(d_{Z}(1,w_n)\geq K)\leq Cl \cdot n$ for each $n\geq 0$. In fact $\matP(d_{Z}(1,w_n)\geq K)$ is the expected value of the indicator function of the event $A_Z=\{d_{Z}(1,w_n)\geq K\}$, so that the aforementioned sum equals the expected value of the random variable $|\{Z\in\mathcal S : d_Z(1,w_n)\geq K\}|$. By Proposition \ref{lowerbound} we have $|\{Z\in\mathcal S : d_Z(1,g)\geq K\}|\leq C d_G(1,g)$ for each $g\in G$, hence the estimate follows.

Choose any $A>M$. For $n$ large enough that $A\log n\geq K$ we have
 \begin{align*}
 \matP\big(\exists Z\in\mathcal S\ :\ & d_{Z}(1,w_n) \geq A\log n \big) \\
&\le \sum_{Z\in\mathcal S}\matP\big( d_{Z}(1,w_n)\geq A\log n \: | \: d_{Z}(1,w_n)\geq K \big)\matP(d_{Z}(1,w_n)\geq K) \\
 & \le Me^{-(A\log n-K)/M}\sum \matP\big(d_{Z}(1,w_n)\geq K\big) \\
 &\le (MCl e^{K/M}) n^{1-A/M},
 \end{align*}
where the second inequality follows from Proposition \ref{prop:exp_decay_proj}. Since this quantity tends to $0$ by our choice of $A$, the proof is complete.
\end{proof}

\section{Proof of Theorem \ref{thm:examples}} \label{sec:examples} Here we show how Theorem \ref{thm:examples} follows from Theorem \ref{thm:main}. We give full details in the case of mapping class groups; the other cases follow a very similar outline and we provide slightly fewer details and the needed references.

\subsection{Mapping class groups and subsurface projections} \label{sec:mcg}
In this subsection, we assume that the reader has some familiarity with subsurface projections, as defined by Masur--Minsky in \cite{MM2}. Set $G = \Mod(S)$. Since there are only finitely many $G$--orbits of (isotopy classes of) essential subsurfaces, we can deal with each orbit separately. For any proper essential subsurface $Y_0$ of $S$, let $\mathcal{S} = \{gY_0 : g \in G\}$, and for $Z \in \mathcal{S}$, let $\pi_Z$ denote the subsurface projection to the curve complex $\C(Z)$ of $Z$. If we complete $\partial Y_0$ to a marking $\lambda$ (a collection of curves cutting $S$ into disks and once-punctured disks), then we define $\pi_Z \colon G \to \C(Z)$ by $\pi_Z(g) = \pi_Z(g\lambda)$. With this definition, $\pi_Z$ is $L$--Lipschitz for some $L >1$ (\cite[Lemma 2.5]{MM2}). This verifies the first $3$ conditions in the definition of a projection system (Definition \ref{defn:proj_system}). 

Subsurfaces $Y$ and $Z$ overlap if, up to isotopy, they are neither disjoint nor nested. In this case, we write $Y \pitchfork Z$. Note that by construction, if $Y$ and $Z$ overlap, then $d_Z(\lambda, Y) \le L$. Here, $\pi_Z(Y)$ is by definition the projection of the boundary of $Y$ to $\C(Z)$.
The usual Behrstock inequality \cite{Be} states that there exists $D\ge 0$ such that for any marking $\eta$ and overlapping subsurfaces $Y$ and $Z$,
$\min\{d_Y(\partial Z,\eta), d_Z(\partial Y,\eta)\} \le D.$ Hence, setting $B = D+L$ verifies the fourth condition of Definition \ref{defn:proj_system}.
To show that subsurface projections give a projection system for the mapping class group, it remains to show the there is a bound on the size of a collection of pairwise nonoverlapping subsurfaces of $S$. This fact is easily verified since such subsurfaces can be realized simultaneously as either disjoint or nested on $S$ (see the proof of Theorem 6.10 in \cite{MM2}).

Let $\mu$ be a symmetric probability measure on $\Mod(S)$ whose support is finite and generates $\Mod(S)$. We show that there is a $k$ such that $\overline{\mu} = \mu^{*k}$ is admissible. Here, $\mu^{*k}$ is the $k$-fold convolution power of $\mu$, which by definition is
\[
\mu^{*k}(g)  = \sum_{x_1\ldots x_k = g}\mu(x_1)\cdots \mu(x_k). 
\]
From this, Theorem \ref{thm:examples} for $G = \Mod(S)$ follows immediately from Theorem \ref{thm:main} (and the fact that the projections are $L$--Lipschitz). Let $Y_0$ and $\lambda$ be as above and fix $f \in \Mod(S)$ which acts as a pseudo-Anosov on $Y_0$ such that $f^2$ has translation length at least $2B$ in $\mathcal C(Y_0)$. Further, fix a pseudo-Anosov $g \in \Mod(S)$ whose translation length on $\C(S)$ is at least $3L +1$.
Choose $k$ so that, up to replacing $f$ and $g$ with appropriate powers, $f$ and $g$
 are in the support of $\overline{\mu} = \mu^{*k}$. Note we also have that $f^{-1},g^{-1} \in \mathrm{supp}(\overline \mu)$.
 
Now the first condition of Definition \ref{defn:admissible} hold for $h_1 = g^{-1}, h_2 = g$ since if both $gY_0$ and $g^{-1}Y_0$ fail to overlap some surface $Z$ and $\gamma$ is a boundary component of $Y_0$, then 
\[
d_S(\gamma, g^2\gamma) = d_S(g^{-1}\gamma, g\gamma) \le d_S(g^{-1}\gamma,\partial Z) + d_S(\partial Z, g\gamma) \le 3L,
\]
contradicting our choice of $g$. 
The second condition of Definition \ref{defn:admissible} is satisfied for $x_1=f,x_2=f^{-1}$, while the third condition holds since the support of $\overline{\mu}^{*n}$ contains $f^n$. The final condition of Definition \ref{defn:admissible} is easily deduced from the fact that, by \cite[Theorem 1.2]{Maher-explinprog}, the random walk makes linear progress with exponential decay in $\mathcal C(S)$.

\subsection{Quasiconvex subgroups}

Consider the setup of a quasiconvex subgroup $H$ of a hyperbolic group $G$, with $\mathcal S$ the family of its cosets. The first 3 properties in Definition \ref{defn:proj_system} are easy. The relation $\pitchfork$ is having bounded projection onto each other, so that the fourth item follows from basic hyperbolic geometry. The final item follows from finiteness of width \cite{GMRS}.

As in the case of mapping class groups, the proof of admissibility uses a hyperbolic space, $X$, that we now define. We let $S$ be any finite generating set of $G$, and consider the Cayley graph $X$ of $G$ with respect to the infinite generating set $S\cup H$.

We could not find a reference for the following statement, despite it being implicit in several papers. Recall that an action on a hyperbolic space is non-elementary if there exist two loxodromic elements with no common limit point at infinity.

\begin{lemma}
 $X$ is hyperbolic and the action of $G$ on $X$ is non-elementary.
\end{lemma}

\begin{proof}
 By \cite[Theorem 6.4]{Gersten}, $X$ is hyperbolic. In fact, by \cite[Proposition 2.6]{KapovichRafi} quasigeodesics in $G$ map to unparameterized quasigeodesics in $X$. In particular, if $g$ is any element of $G$, then the action of $g$ on $X$ is either elliptic or loxodromic. The combination of \cite[Theorem 1-(b)]{Minasyan:products} (which provides an element $x\in G$ not conjugate into $H$) and \cite[Theorem 2]{Minasyan:products} (for $K=\langle x\rangle$) prove that there exists an element of $G$, that we denote $g$, that cannot act elliptically, and hence it acts loxodromically. Similarly, we can apply the same reasoning to find an element $g'$ that acts loxodromically on the Cayley graph of $G$ with respect to $S\cup H\cup\langle g\rangle$, and in particular it will also act loxodromically on $X$. It is easy to see that $g$ and $g'$ cannot have a common limit point at infinity.
\end{proof}

A convolution power of $\mu$ will have support containing an infinite order element of the quasiconvex subgroup, as well as an element with large translation distance on $X$, easily implying the first three items of Definition \ref{defn:admissible}.

We can now apply the linear progress result in \cite{MaherTiozzo} to get that the random walk we are considering makes linear progress with exponential decay in $X$, easily implying the final condition of Definition \ref{defn:admissible} (since if two cosets are far away in $X$ then they are, in particular, far away in $G$ and hence they have bounded projection onto each other).

\subsection{Peripheral subgroups}

Consider the setup of a peripheral subgroup with an undistorted element of a relatively hyperbolic group, with $\mathcal S$ the family of its cosets. The relation $\pitchfork$ in this case is just being distinct, and the needed properties of projections follow from, e.g., \cite{Si-proj}.

 The last item once again uses a hyperbolic space. In this case, the hyperbolic space $X$ is a coned-off graph: If $S$ is a finite generating set for $G$, then $X$ is the Cayley graph with respect to the generating set $S\cup H_1\cup\dots\cup H_n$. The hyperbolicity of $X$ is part of the definition of relative hyperbolicity from \cite{Farb}, and the action is non-elementary due to results in \cite{Os-elem}. More precisely, \cite[Lemma 4.5]{Os-elem} gives a loxodromic element $g$ for the action. Moreover, $g$ is contained in an elementary subgroup $E(g)$ that can be added to the list of peripheral subgroups \cite[Corollary 1.7]{Os-elem}, so that one can find a loxodromic element $g'$ with respect to the new list of peripherals. Similarly to the hyperbolic group case, it is easy to see that $g,g'$ are the required pair of loxodromic elements.

Admissibility now follows similarly to the other cases, using a sufficiently large power of the undistorted element of $H$ (any undistorted element of a group $H$ has powers with arbitrarily large translation distance on the Cayley graph of $H$).

\subsection{$\Out(F_n)$ and subfactor projections}
Subfactor projections were introduced by \linebreak Bestvina--Feighn in \cite{BFproj} and refined in \cite{Taylproj}. We refer to these references for definitions and details. 

For a rank $\ge 2$ free factor $Y_0$ of $F_n$, let $\pi_{Y_0}$ denote the subfactor projection to $\C(Y_0)$, the free factor graph of $Y_0$. 
Set $G = \Out(F_n)$ and let $\mathcal{S} = \{g Y_0 : g \in G\}$. Finally, let $\lambda$ be a $F_n$--marked graph, i.e. a graph with a fixed isomorphism $F_n \to \pi_1(\lambda)$, containing a subgraph $\lambda_{Y_0}$ with $\pi_1(\lambda_{Y_0}) = Y_0$. 

For any $Z \in \mathcal{S}$, define $\pi_Z \colon G \to \C(Z)$ by $\pi_Z(g) = \pi_Z(g\lambda)$. Here, $g \lambda$ denotes the image of $\lambda$ under $g \in G$ with respect to the natural left action of $G$ on the set of marked graphs. Free factors $X$ and $Z$ of $F_n$ are said to be disjoint if, up to conjugation, $F_n = W*X*Z$ for some (possibly trivial) free factor $W$. The factors $X$ and $Z$ overlap, written $X \pitchfork Z$, if they are neither disjoint nor nested. By \cite[Theorem 1.1]{Taylproj}, when free factors $X$ and $Z$ overlap, there is a well-defined coarse projection $\pi_Z(X) \subset \C(Z)$ and that the natural version of the Behrstock inequality holds (see also \cite[Proposition 4.18]{BFproj}). This, together with the fact that $d_Z(Y_0, \lambda)$ is bounded for all $Z \in \mathcal{S}$ with $Y_0 \pitchfork Z$, implies the fourth condition of Definition \ref{defn:proj_system}. Finally, condition $(5)$ follows, for example, from \cite[Lemma 4.14]{BFproj}. Hence, subfactor projections form a projection system.

From this, it follows just as in the situation of $\Mod(S)$ that if $\mu$ is a symmetric probability measure on $\Out(F_n)$ whose support is finite and generates $\Out(F_n)$, then $\overline \mu = \mu^{*k}$ is admissible for some $k\ge1$. The only needed modifications are that one chooses $g$ to be fully irreducible with large translation length on the free factor complex of $F_n$,
chooses $f$ to fix $Y_0$ and have large translation length on $\C(Y_0)$, and applies the general linear progress result of \cite[Theorem 1.2]{MaherTiozzo}.

\section{Shortest curves in random mapping tori} \label{sec:shortest}
Let $S$ be a closed connected orientable surface of genus at least $2$. Hereafter, we assume that the reader is familiar with the subsurface projection machinery of \cite{MM2} and refer them to the terminology established in Section \ref{sec:mcg}.
Throughout this section, we fix a symmetric measure $\mu$ whose support is finite and generates $\Mod(S)$, and let $(w_n)$ denote the random walk driven by $\mu$. In this setting, we have that for almost every sample path $(w_n$), $w_n$ is pseudo-Anosov for sufficiently large $n\ge 0$ \cite{Maher-explinprog}. We will use this fact freely without further comment.

For a pseudo-Anosov $f \in \Mod(S)$, let $\lambda^+(f)$ and $\lambda^-(f)$ denote its stable and unstable laminations. For a subsurface $Y$ of $S$, define
\[
d_Y(f) = d_Y(\lambda^+(f), \lambda^-(f)),
\]
to be the distance in the curve graph of $Y$ between the projections of the stable and unstable laminations of $f$. Note that by invariance of the laminations, $d_{f^iY}(f) = d_Y(f)$ for all $i \in \Z$.
See \cite{MM2, ECL1} for details. When $Y$ is an annulus about the curve $\alpha$, we use the notation $d_\alpha$ rather than $d_Y$. For a curve $\alpha$ in $S$, let $\mathcal{Y}_\alpha$ be the (nonannular) subsurfaces with $\alpha$ as a boundary component. Finally, in the following proposition we will need to make use of the \emph{bounded geodesic image theorem} of \cite{MM2}. This states that there is a constant $M >0$ such that if $\gamma$ is a geodesic in the curve graph $\C(S)$ that does not meet the $1$--neighborhood of $\partial X$ for some subsurface $X$ of $S$, then the diameter of the projection of $\gamma$ to $\C(X)$ is at most $M$.

When writing expressions such as $\sum_{Z \subset X}$, for $X$ a subsurface, we mean that we are summing over all isotopy classes of (essential) subsurfaces of $X$. Recall that the notation $\{\{x\}\}_L$ denotes $x$ if $x\geq L$ and $0$ otherwise. Using Theorem \ref{thm:examples} for mapping class groups, we can prove the following: 

\begin{prop} \label{prop:pA_sub_growth}
For the random walk $(w_n)$ on $\Mod(S)$ as above, there are $K_2, C_2 \ge 1$ such that 
\begin{enumerate}
\item 
$\matP \left( \sup_{X \subsetneq S } \sum_{Z \subset X} \{\{d_Z(w_n)\}\}_{K_2} \le C_2 \cdot \log(n) \right) \to 1, \: \text{ as } n \to \infty.$

\item 
$\matP \left( \exists \: \mathrm{a} \: \mathrm{curve}
\: \alpha_n \: \mathrm{in} \; S : d_{\alpha_n} (w_n) \ge C_2^{-1} \cdot \log(n) \; \mathrm{and} \; 
\sup_{X \in \mathcal{Y}_{\alpha_n}} d_X(w_n) \le K_2 \right) \to 1, \linebreak
\text{ as } n \to \infty.$
\end{enumerate}
\end{prop}

\begin{proof}
The main idea of the argument is that the quantities $d_Y(w_n)$ are typically closely related to the distances $d_Y(1,w_n)$, as studied in the first part of the paper. Hence, we first prove the corresponding statements for $d_Y(1,w_n)$, and then we will translate them into statements about $d_Y(w_n)$.
Just as in Section \ref{sec:mcg}, it suffices to assume that the measure $\mu$ is admissible.

Item 1 for the $d_Y(1,w_n)$ is  \cite[Lemma 5.12]{Si-tracking}.

Item 2 for the $d_Y(1,w_n)$ is a slight variation of Theorem \ref{thm:lower}, and can be shown by modifying the proof as follows. We fix $Y$ to be an annular subsurface with core curve $\alpha$. For any subsurface $Z \subset S$ (including the case $Z =S$), $\pi_Z(g) \subset \C(Z)$ is the subsurface projection of $g \lambda$, where $\lambda$ is a fixed marking chosen to contain $\alpha$. As in Section \ref{sec:mcg}, we have $\mathrm{diam}(\pi_Z(g))\leq L$ for all $Z\subseteq S$.
Recall that the first step we took in Section \ref{sec:lowerbound} is to choose some $x_n\in G = \Mod(S)$ whose most important property is that $d_Y(1,x_n)\geq \epsilon_2\log(n)$. In this case, we want 3 large projections instead of 1, and 2 of them will serve as ``buffer'' for the ``middle'' one. We now fix once and for all $g \in\Mod(S)$ so that $d_{\mathcal C(S)}(x,gx)\geq 10$ for all $x \in \C(S)$ (in particular $g^{\pm 1} Y' \pitchfork Y$ for $Y'=Y$ or $Y'$ disjoint from $Y$), and choose $x_n=s_ngt_ngu_n$ so that
\begin{enumerate}
\item $p_n:=\matP(w_k=x_n)\geq n^{\epsilon_1\log c}$,
 \item $s_n,t_n,u_n$ are powers of Dehn twists around $\alpha$,
 \item $d_Y(1,s_n),d_Y(1,t_n),d_Y(1,u_n)\geq \epsilon_2\log n$, for some fixed $\epsilon_2>0$.
\end{enumerate}

We then define $W_i,L_i,R_i,Y_i$ as in Subsection \ref{subsec:proof_lower}. Similarly to the claim below Figure \ref{Fig:subpaths}, one can show that if $Y_i=1$ and $n$ is sufficiently large, then $d_Z(1,w_n)$ is $\log$--large whenever $Z$ is one of the annuli $w_iY,w'_iY=w_is_ngY$, or $w_is_ngt_ngY$, and that, moreover, all projection distances to subsurfaces $X$ contained in the complement of $w'_iY$ are uniformly bounded. In fact, we are in the situation where $w_n$ can be written as a product of group elements $g_1g_2\dots g_7$ (where $g_1 = w_i, \: g_2=s_n,\: g_3=g,\: g_4=t_n, \: g_5=g, \: g_6=u_n$) so that for $k$ odd we have that $g_kY\pitchfork Y$ and $g_k$ has controlled projection on $Y$, while for $k$ even we have that $g_k$ is a large power of a Dehn twist supported of $Y$. Just as before, in this situation the projections created by the even $g_k$ ``persist.'' Similarly, if $X$ is contained in the complement of $w'_iY=g_1g_2g_3Y$, then, first of all, $X \pitchfork g_1Y$ and $X \pitchfork g_1\dots g_5Y$ because of our hypothesis on $g$. 
Then, by the triangle inequality, 
\[
d_X(1,w_n) \le d_X(1,g_1) + d_X(g_1 , g_1\ldots g_5) + d_X(g_1\ldots g_5, w_n)+ 2L,
\]
where the middle term is at most $d_{X'}(g^{-1},g) + 2L$, for $X'$ a subsurface contained in the complement of $Y$. Indeed, since $X$ is disjoint from $g_1g_2g_3Y$, $X' = (g_1g_2g_3)^{-1}X$ is disjoint from $Y$. Hence, $d_X(g_1 , g_1\ldots g_5) = d_{X'}(g_3^{-1}g_2^{-1}, g_4g_5) = d_{X'}(g^{-1},g)$, where we have used that $g_2$ and $g_4$ are twists about $\alpha$ and that $\alpha \subset \lambda$, so that $\mathrm{diam}(\pi_{X'}(\lambda)\cup \pi_{X'}(g_2^{-1}\lambda))\leq 2L$ and $\pi_{X'}(g_4g_5)=g_4\pi_{X'}(g_5)=\pi_{X'}(g_5)$.
Since this middle term is uniformly bounded, we just need to bound the other two terms. We will focus on the first term, as the last term is handled similarly. 
Now if $d_X(1,g_1)$ is larger than $B+L$ (where $B$ is from the Behrstock inequality stated in Section \ref{sec:mcg}) then 
$d_X(1,g_1 \partial Y) $ is larger than $B$ and so $d_{g_1Y}(1,\partial X) \le B$. But $\partial X$ and $w_i'\partial Y$ are disjoint and so the Lipschitz property of subsurface projections implies that $d_{g_1Y}(1,w_i') \le B+L$, contradicting that the log--large projection to $g_1Y$ persists.

With our current setup, Lemma \ref{lem:exp1} holds as stated, for the same reason that  $W_i,L_i,R_i$ are independent, and Proposition \ref{prop:exp2} also holds with a similar proof based on alternating small and large projections as above. These are all the needed ingredients to conclude the proof that, given a large $n$, with high probability there exists $i$ with $Y_i=1$, as we did for Theorem \ref{thm:lower}. 

We now translate the statement for the $d_Y(1,w_n)$ into a statement for the $d_Y(w_n)$. 
We being by showing that, with probability going to $1$, for any proper subsurface $X$ of $S$, we have $d_X(w_n) \le 3 \sup_i d_{w_n^iX}(1,w_n)+ 2M+2L$, were $M$ is the constant from the bounded geodesic image theorem. 
This in particular will prove item $1$. 

We regard $w_n$ as a product of two shorter random walks $u_n$, $v_n$, of length approximately $n/2$ each. As explained in, e.g., \cite[Lemma 23]{MaherSisto}, the following is a consequence of results in \cite{Maher} and \cite{MathieuSisto}. With probability going to 1 as $n$ goes to infinity, 
\begin{itemize}
\item $w_n$ acts hyperbolically on $\mathcal C(S)$ with (quasi-)axis $A_n$,
\item writing  $\gamma = [\pi_S(1),\pi_S(w_n)]$, $\gamma'=[\pi_S(1),\pi_S(u_n)]$, and $\gamma''= u_n \cdot [\pi_S(1),\pi_S(v_n)] $, the axis $A_n$ has Hausdorff distance $O(\log n)$ from both 
\[
\bigcup w_n^k \cdot \gamma \quad \text{and} \quad \bigcup w_n^k \cdot (\gamma' \cup \gamma''),
\]
where with a slight abuse of notation we regarded the various $\pi_S(\cdot)$ as vertices of $\mathcal C(S)$,
\item and each of $\gamma$, $\gamma'$, and $\gamma''$ have length at least $\epsilon' n$ for some uniform $\epsilon' >0$. 
\end{itemize}
Going forward, we will assume that $w_n$ has this form.

Now if $d_X(w_n) \ge 2M$, then the bounded geodesic image theorem implies that $\partial X$ lies within a uniformly bounded distance from the axis $A_n$. Hence, there is an $i \in \Z$ such that the boundary of $X_i= w_n^i X$ has bounded distance from $\gamma$, and we see that
\begin{align*}
d_X(w_n) &= d_{X_i}(w_n) \\
&\le d_{X_i}(\lambda^-(w_n), w_n^{-1}) + d_{X_i}(w_n^{-1},w_n^2) + d_{X_i}(w_n^2,\lambda^+(w_n))+2L \\
&\le M + 3 \sup_i d_{X_i}(1,w_n) + M+2L,
\end{align*}
where the last inequality follows from the triangle inequality and another applications of the bounded geodesic image theorem. In a bit more detail, with our setup, the geodesics from $w_n^{-1}$ to $\lambda^-(w_n)$ and from $w_n^2$ to $\lambda^+(w_n)$ have $\C(S)$--distance from $\partial X_i$ growing linearly in $n$ since the length of $\gamma$ is greater than $\epsilon'n$ and $d_{\C(S)}(w_n^i ,A_n) = O(\log(n))$ for $i\in \Z$. Hence, the bounded geodesic image theorem implies that, with probability going to $1$, their images in $\C(X_i)$ have diameter at most $M$.

Let us now prove item 2. Using what we have already shown applied to the random walk $u_n$, with probability going to 1 as $n$ goes to infinity, there is a curve $\alpha$ on $S$ such that for some $1\le i < n$ and $\epsilon, K>0$:
\begin{itemize}
 \item the axis of $w_n$ has the form described above,
 \item for $h\in \{u_i,u_is_ng, u_is_ngt_ng\}$, each projection $d_{h\alpha}(1,u_n)$ is greater than $\epsilon\log (n)$ and $d_{X}(1,u_n)\leq K$ for all $X$ disjoint from $u_is_ng \cdot \alpha$, and 
 \item for $h\in \{u_i,u_is_ng, u_is_ngt_ng\}$, each projection $d_{h\alpha}(v_n^{-1},1)$ and $d_{h\alpha}(u_n,w_n)$ is bounded by $\epsilon/10 \log(n)$.
\end{itemize}

The third item holds because of an application of Proposition \ref{prop:exp_decay_proj} and a simple conditioning argument. We explain the bound on the term $d_{h\alpha}(u_n,w_n)$ and leave the other to the reader.  Let $\mathcal Y_n=\{f \in\Mod(S): \exists Y=Y(f){\rm \ annulus\ with\ }  d_Y(1,f)\geq\epsilon\log n\}$. The probability that the third item holds is at least
\begin{align*}
&\sum_{f\in \mathcal Y_n} \matP(d_{Y(f)}(u_n,w_n)\leq \epsilon/10 \log (n))\matP(u_n=f) \\
&=
\sum_{f\in \mathcal Y_n} \matP(d_{f^{-1}Y(f)}(1,v_n)\leq \epsilon/10 \log(n))\matP(u_n=f)
\end{align*}
The terms $\matP(d_{f^{-1}Y(f)}(1,v_n)\leq \epsilon/10 \log n)$ are arbitrarily close to $1$ as $n$ goes to $\infty$ by Proposition \ref{prop:exp_decay_proj}, and $\sum_{f\in \mathcal Y_n} \matP(u_n=f)=\matP(u_n\in\mathcal Y_n)$ goes to $1$ by the version of item 2 that we proved above (if $\epsilon$ is sufficiently small).

We are now ready to conclude: by the second and third bullet points and the triangle inequality, for any $n$ large enough we have
\[
d_{h\alpha}(v_n^{-1},w_n) \ge 8\epsilon/10 \log(n) - 2L,
\]
for each $h\in\{u_i,u_is_ng, u_is_ngt_ng\}$ with probability going to $1$, as $n \to \infty$. Just as in our previous application of the bounded geodesic image theorem, since $\gamma''$ has length at least $\epsilon' n$, each of $h\alpha$ (for $h$ as above) is at bounded distance from $\gamma'$, and the distances from both $v_n^{-1}$ and $w_n$ to the axis $A_n$ is $O(\log(n))$, we get that $d_{h\alpha}(w_n) \ge  8\epsilon/10 \log(n) - 2L - 2M$. Now exactly as in the first part of the proof, the $3$ large projections to $h \alpha$ for $h \in  \{u_i,u_is_ng, u_is_ngt_ng\}$ guarantee that $d_X(w_n)$ is uniformly bounded whenever $X$ is a subsurface of $S$ disjoint from $u_is_ng \cdot \alpha$. This completes the proof.
\end{proof}

Informally, Proposition \ref{prop:pA_sub_growth} controls the size of subsurface distances along the axis of the random pseudo-Anosov $w_n$. These distances in turn control the lengths of curves in the hyperbolic manifold $M(w_n)$.

\subsection*{Length estimates}
Following Minsky \cite{ECL1}, we set
\[
\omega_\alpha(f) = d_\alpha(f) + i \cdot \big(1 + \sum_{Y \in \mathcal{Y}_\alpha}\{\{ d_Y(f) \}\}_{K_3} \big),
\]
for some constant $K_3 \ge 0$ as determined in \cite{ECL1}. By \cite[Section 9.5]{ECL1}, we may suppose that $K_3 \ge K_2$ for $K_2$ as in Proposition \ref{prop:pA_sub_growth}.

For a pseudo-Anosov $f \in \Mod(S)$, the complex length of $\alpha$ in the hyperbolic mapping torus $M(f)$ is
\[
\lambda_\alpha(f) = \ell_\alpha(f) +  i \cdot \theta_\alpha(f),
\]
where $\ell_\alpha(f)$ is the usual hyperbolic length of $\alpha$ in $M(f)$ and $\theta_\alpha(f) \in (-\pi,\pi]$ is the rotational part of $\alpha$. Recall that we are interested in smallest $l_\alpha (w_n)$ along the random walk $(w_n)$.


We regard $\omega_\alpha(f)$ and $2\pi i / \lambda_{\alpha}(f)$ as points in the upper-half plane model of the hyperbolic plane $\mathbb H$. 
The following theorem is part of the Brock--Canary--Minsky proof of Thurston's Ending Lamination Conjecture \cite{ELC2}:

\begin{thm}[Length Bound Theorem \cite{ELC2}] \label{th:ELC}
There are $D,\epsilon \ge 0$, depending only of $S$, such that for any pseudo-Anosov $f \in \Mod(S)$ and any curve $\alpha$ in $S$ with $\ell_\alpha(f) \le \epsilon$:
\[
d_{\mathbb{H}} \left ( \omega_\alpha(f), \frac{2\pi i}{\lambda_\alpha(f)} \right ) \le D.
\]
Moreover, if $|\omega_\alpha| \ge M$, for $M\ge0$ depending only on $S$, then $\ell_\alpha(f) \le \epsilon$.
\end{thm}

\begin{cor} \label{cor:useful_ELC}
With notation as above and the hypotheses of Theorem \ref{th:ELC}, there is a $D_1 \ge0$ such that
\begin{align}\label{eq:length_est}
D_1^{-1} \cdot \frac{2 \pi \: S_\alpha(f)}{d_\alpha^2(f) + S_\alpha^2(f)}. \le \ell_\alpha(f) \le D_1 \cdot \frac{2 \pi \: S_\alpha(f)}{d_\alpha^2(f) + S_\alpha^2(f)},
\end{align}
where $S_\alpha(f) = 1 + \sum_{Y \in \mathcal{Y}_\alpha}\{\{ d_Y(f) \}\}_{K_3}$.
\end{cor}

\begin{proof}
Using Theorem \ref{th:ELC} together with the $\mathbb{H}$ isometry $z \to -1/z$, we have
\[
d_{\mathbb{H}} \left (\frac{-1}{\omega_\alpha(f)}, \frac{i \cdot \lambda_\alpha(f)}{2\pi} \right ) \le D.
\]
However, if $d_{\mathbb{H}}(z_1,z_2) \le D$, then $|\log (\frac{\Im z_1}{\Im z_2})| \le D$. 
Since $\Im (-1/\omega_{\alpha}(f)) = S_\alpha / (d_\alpha^2(f) + S_\alpha^2(f))$ and  
$\Im(i \cdot \lambda_\alpha(f) / 2\pi) =  \ell_\alpha(f)/2\pi$,
setting $D_1 = e^D$ completes the proof.
\end{proof}

\subsection*{Proof of Theorem \ref{thm:shortest_geo}}
We break the proof into $3$ steps:

 \textbf{Step 1:} There is a constant $C_3>0$, depending only on $S$, such that with probability going to $1$, there is a curve $\alpha_n$ in $S$ with $\ell_{\alpha_n}(w_n) \le \frac{C_3}{\log^2(n)}$.
 
To see this, note that by item $(2)$ of Proposition \ref{prop:pA_sub_growth}, with probability approaching $1$, there is a curve $\alpha_n$ such that $d_{\alpha_n} (w_n) \ge C_2^{-1} \cdot \log(n)$ and $\sup_{Y \in \mathcal{Y}_{\alpha_n}} d_Y(w_n) \le K_2$. By Corollary \ref{cor:useful_ELC}, 
\begin{align} \label{eq:short}
\ell_{\alpha_n}(w_n) \le D_1 \cdot \frac{2\pi}{d^2_\alpha + 1} \le D_1 \cdot \frac{2\pi}{C_2^{-1} \cdot \log^2(n) + 1},
\end{align}
where we have used that $K_3 \ge K_2$. This completes Step $1$.
\\

 \textbf{Step 2:} There is a constant $C_4>0$, depending only on $S$, such that with probability going to $1$, $\ell_{\alpha}(w_n) \ge \frac{C_4}{\log^2(n)}$ for any curve $\alpha$ in $S$.
 
By item $(1)$ of Proposition \ref{prop:pA_sub_growth}, for all subsurface $Y$ of $S$, $\sup_{Y \subsetneq S } \sum_{Z \subset Y} \{\{d_Z(w_n)\}\}_{K_2} \le C_2 \cdot \log(n)$ with probability going to $1$. Hence, with probability going to $1$,
\begin{align*}
\ell_\alpha(w_n) &\ge D_1^{-1} \cdot \frac{2 \pi \: S_\alpha(w_n)}{d_\alpha^2(w_n) + S_\alpha^2(w_n)} \\
&\ge C_2^{-2}D_1^{-1} \cdot \frac{2 \pi }{\log^2(n)}. 
\end{align*}
This finishes Step $2$.
\\
 
\textbf{Step 3:} With probability going to $1$, $\mathrm{sys}(M(w_n))$ is realized by a geodesic $g_n$ that is isotopic to a curve $\beta_n$ in the fiber $S$.

By Step $1$, we know that, with probability approaching $1$, $\mathrm{sys}(M(w_n)) \to 0$. Hence, it suffices to show that any sufficiently short geodesic in a hyperbolic fibered $3$-manifold $M$ with fiber $S$ is isotopic to a curve in $S$. This is an easy consequence of \cite[Lemma 2.1]{BiringerSouto}, but an older argument that uses well-known results in the theory of Kleinian groups was already known to experts. We now sketch the argument.

First, any primitive geodesic $g$ which is sufficiently short (depending only on $S$) must be homotopic into $S$. Otherwise, $g$ meets the image of any map $f \colon S \to M$ homotopic to the fiber. (In fact, $g$ has nonzero algebraic intersection number with $S$.) In particular, $g$ meets the image of a pleated surface $f \colon S \to M$ (see \cite{thurston1998hyperbolic, thurstonnotes}). As in \cite[Section 3.22]{ECL1}, for any $\epsilon_1 \ge0$ there is an $\epsilon_2 <\epsilon_1$ such that
\[
f(S_{[\epsilon_1, \infty)}) \subset M_{[\epsilon_2, \infty)},
\] 
where $N_{[\epsilon,\infty)}$ is the $\epsilon$-thick part of $N$, i.e. where the injectivity radius is greater than $2\epsilon$. Fix $\epsilon_1$ less than the $3$-dimensional Margulis constant. (See \cite{thurstonnotes}).
If $g$ has length less than $\epsilon_2$, then by the thick-thin decomposition of $S$, there is a closed curve $\gamma$ on $S$ with length less than $\epsilon_1$ whose image in $M$ meets $g$. The Margulis Lemma \cite[Chapter 5.10.1]{thurstonnotes} implies that $g$ and $\gamma$ (at their intersection point) generate a virtually cyclic subgroup $\pi_1(M)$. This, however, contradicts that $\gamma$ and $g$ cannot be homotopic up to powers.

Now since $g$ is homotopic into $S$, it lifts to the cover $\widetilde M$ corresponding to $S$. Since $\widetilde M \cong S \times \R$, a theorem of Otal \cite{otal1995nouage} implies that if $g$ is sufficiently short, then it is isotopic to a curve $\beta$ in $S$. This completes the proof of Step $3$ and of the theorem.
\qed

\bibliographystyle{abbrv}
\bibliography{large_proj.bbl}

\begin{thebibliography}{10}

\bibitem{ATW}
T.~Aougab, S.~J. Taylor, and R.~C. Webb.
\newblock Effectivizing the {M}asur--{M}insky distance formulas and
  applications to hyperbolic $3$-manifolds.
\newblock In preparation, April 2017.

\bibitem{Be}
J.~Behrstock.
\newblock Asymptotic geometry of the mapping class group and {T}eichm\"uller
  space.
\newblock {\em Geom. Topol.}, 10:1523--1578, 2006.

\bibitem{behrstock2014hierarchically}
J.~Behrstock, M.~F. Hagen, and A.~Sisto.
\newblock Hierarchically hyperbolic spaces {I}: curve complexes for cubical
  groups.
\newblock {\em To appear in Geom. Topol.}, 2014.

\bibitem{behrstock2015hierarchically}
J.~Behrstock, M.~F. Hagen, and A.~Sisto.
\newblock Hierarchically hyperbolic spaces {II}: combination theorems and the
  distance formula.
\newblock {\em arXiv:1509.00632}, 2015.

\bibitem{BFproj}
M.~Bestvina and M.~Feighn.
\newblock Subfactor projections.
\newblock {\em J. Topol.}, 2014.

\bibitem{BiringerSouto}
I.~Biringer and J.~Souto.
\newblock Ranks of mapping tori via the curve complex.
\newblock {\em To appear in J. Reine Angew. Math.}, 2017.

\bibitem{ELC2}
J.~Brock, R.~Canary, and Y.~Minsky.
\newblock The classification of {K}leinian surface groups, {II}: {T}he ending
  lamination conjecture.
\newblock {\em Ann. of Math.}, 176(1):1--149, 2012.

\bibitem{calegari2015statistics}
D.~Calegari and J.~Maher.
\newblock Statistics and compression of scl.
\newblock {\em Ergodic Theory and Dynamical Systems}, 35(01):64--110, 2015.

\bibitem{DGO}
F.~Dahmani, V.~Guirardel, and D.~Osin.
\newblock Hyperbolically embedded subgroups and rotating families in groups
  acting on hyperbolic spaces.
\newblock {\em Mem. Amer. Math. Soc.}, 245(1156):v+152, 2017.

\bibitem{Farb}
B.~Farb.
\newblock Relatively hyperbolic groups.
\newblock {\em Geom. Funct. Anal.}, 8(5):810--840, 1998.

\bibitem{Gersten}
S.~M. Gersten.
\newblock Subgroups of word hyperbolic groups in dimension {$2$}.
\newblock {\em J. London Math. Soc. (2)}, 54(2):261--283, 1996.

\bibitem{GMRS}
R.~Gitik, M.~Mitra, E.~Rips, and M.~Sageev.
\newblock Widths of subgroups.
\newblock {\em Trans. Amer. Math. Soc.}, 350(1):321--329, 1998.

\bibitem{KapovichRafi}
I.~Kapovich and K.~Rafi.
\newblock On hyperbolicity of free splitting and free factor complexes.
\newblock {\em Groups Geom. Dyn.}, 8(2):391--414, 2014.

\bibitem{Maher}
J.~Maher.
\newblock Random walks on the mapping class group.
\newblock {\em Duke Math. J.}, 156(3):429--468, 2011.

\bibitem{Maher-explinprog}
J.~Maher.
\newblock Exponential decay in the mapping class group.
\newblock {\em J. Lond. Math. Soc. (2)}, 86(2):366--386, 2012.

\bibitem{MaherSisto}
J.~Maher and A.~Sisto.
\newblock Random subgroups of acylindrically hyperbolic groups and hyperbolic
  embeddings.
\newblock arXiv:1701.00253.

\bibitem{MaherTiozzo}
J.~Maher and G.~Tiozzo.
\newblock Random walks on weakly hyperbolic groups.
\newblock {\em To appear in J. Reine Angew. Math.}, 2017.

\bibitem{MM2}
H.~A. Masur and Y.~N. Minsky.
\newblock Geometry of the complex of curves. {II}. {H}ierarchical structure.
\newblock {\em Geom. Funct. Anal.}, 10(4):902--974, 2000.

\bibitem{MathieuSisto}
P.~Mathieu and A.~Sisto.
\newblock Deviation inequalities for random walks.
\newblock arXiv:1411.7865.

\bibitem{Minasyan:products}
A.~Minasyan.
\newblock Some properties of subsets of hyperbolic groups.
\newblock {\em Comm. Algebra}, 33(3):909--935, 2005.

\bibitem{ECL1}
Y.~Minsky.
\newblock The classification of {K}leinian surface groups, {I}: {M}odels and
  bounds.
\newblock {\em Ann. of Math.}, pages 1--107, 2010.

\bibitem{Minsky-bounded_geometry}
Y.~N. Minsky.
\newblock Bounded geometry for {K}leinian groups.
\newblock {\em Invent. Math.}, 146(1):143--192, 2001.

\bibitem{Os-elem}
D.~V. Osin.
\newblock Elementary subgroups of relatively hyperbolic groups and bounded
  generation.
\newblock {\em Internat. J. Algebra Comput.}, 16(1):99--118, 2006.

\bibitem{otal1995nouage}
J.-P. Otal.
\newblock Sur le nouage des g{\'e}od{\'e}siques dans les vari{\'e}t{\'e}s
  hyperboliques.
\newblock {\em Comptes rendus de l'Acad{\'e}mie des sciences. S{\'e}rie 1,
  Math{\'e}matique}, 320(7):847--852, 1995.

\bibitem{otal2001hyperbolization}
J.-P. Otal.
\newblock {\em The hyperbolization theorem for fibered 3-manifolds}, volume~7.
\newblock American Mathematical Soc., 2001.

\bibitem{rivin2008walks}
I.~Rivin.
\newblock Walks on groups, counting reducible matrices, polynomials, and
  surface and free group automorphisms.
\newblock {\em Duke Math. J.}, 142(2):353--379, 2008.

\bibitem{rivin2014statistics}
I.~Rivin.
\newblock Statistics of random 3-manifolds occasionally fibering over the
  circle.
\newblock {\em arXiv:1401.5736}, 2014.

\bibitem{Si-proj}
A.~Sisto.
\newblock Projections and relative hyperbolicity.
\newblock {\em Enseign. Math. (2)}, 59(1-2):165--181, 2013.

\bibitem{Si-contr}
A.~{Sisto}.
\newblock {Contracting elements and random walks}.
\newblock {\em J. Reine Angew. Math., to appear}, 2017.

\bibitem{Si-tracking}
A.~Sisto.
\newblock Tracking rates of random walks.
\newblock {\em Israel J. Math., to appear}, 2017.

\bibitem{Taylproj}
S.~J. Taylor.
\newblock A note on subfactor projections.
\newblock {\em Algebr. Geom. Topol.}, 14(2):805--821, 2014.

\bibitem{thurstonnotes}
W.~P. Thurston.
\newblock The {G}eometry and {T}opology of 3-{M}anifolds.
\newblock {\em Princeton University Notes}, 1980.

\bibitem{thurston1998hyperbolic}
W.~P. Thurston.
\newblock Hyperbolic structures on 3-manifolds, ii: Surface groups and
  3-manifolds which fiber over the circle.
\newblock {\em arXiv preprint math/9801045}, 1998.

\end{thebibliography}
\end{document}